\documentclass[reqno,makeidx,11pt]{amsart}
\usepackage{amssymb,amsfonts,amsmath,graphicx
}
\usepackage[all]{xy}
\def\C{{\mathbb{C}}}

\def\F{{\mathbb{F}}}

\def\N{{\mathbb{N}}}
\def\0{{\mathbb{O}}}
\def\P{{\mathbb{P}}}

\def\Z{{\mathbb{Z}}}

\def\cB{{\mathcal B}}
\def\cC{{\mathcal C}}

\def\cE{{\mathcal E}}

\def\cG{{\mathcal G}}
\def\cH{{\mathcal H}}

\def\cO{{\mathcal O}}

\def\Id{{\rm Id \, }}

\def\IS{{\rm IS }}

\newcommand{\norm}[1]{{\left\|{#1}\right\|}}

\newcommand{\abs}[1]{{\left|{#1}\right|}}
\newcommand{\scal}[1]{{\left\langle{#1}\right\rangle}}
\newcommand{\set}[1]{{\left\{{#1}\right\}}}

\newcommand{\actson}{\curvearrowright}
\newcommand{\rd}{{\,\mathrm d}}

\def\fig{ \centerline{Fig. \the\count200\cGlobal\advance\count200 by 1}}
\newtheorem{thm}{Theorem}[section]
\newtheorem{cor}[thm]{Corollary}

\newtheorem{prop}[thm]{Proposition}

\newtheorem*{thm*}{Theorem}
\newtheorem*{lem*}{Lemma}
\newtheorem*{cor*}{Corollary}

\theoremstyle{definition}
\newtheorem{defn}[thm]{Definition}
\newtheorem{ex}[thm]{Example}
\newtheorem{exs}[thm]{Examples}

\theoremstyle{remark}
\newtheorem{rem}[thm]{Remark}
\newtheorem{rems}[thm]{Remarks}

\title[Some remarks about the weak containment property]{Some remarks about the weak containment property  for groupoids and semigroups}

\author{Claire Anantharaman-Delaroche}
\address{Institut Denis Poisson\newline
\indent Universit\'e d'Orl\'eans, Universit\'e de Tours et CNRS-UMR 7013,\newline
\indent Route de Chartres, B.P. 6759; 
F-45067 Orl\'eans Cedex 2, France}

\email{claire.anantharaman@univ-orleans.fr}

\subjclass[2010]{Primary  43A07 ; Secondary 46L55, 54H20, 22A22, 20M30,20M18}

\keywords{\emph{Groupoids, semigroups, weak containment, amenability}}

\begin{document}
\begin{abstract} A locally compact groupoid is said to have the weak containment property if its full $C^*$-algebra coincides with its reduced one. This property is strictly weaker than amenability and is known to be equivalent to amenability for transformation groupoids relative to actions of exact discrete groups. We believe that for general \'etale groupoids one should have the same equivalence of the two properties under  some mild exactness assumption. In this paper we try to support this statement.
\end{abstract}

\maketitle

\section*{Introduction}
The notion of amenability for locally compact groups takes many forms and is well understood (see \cite{Pat88} for instance). Amenability was introduced in the measured setting for discrete group actions and countable equivalence relations by Zimmer \cite{Zi3, Zi1,Zi2} at the end of the seventies. Soon after, Renault extended this notion to general measured groupoids and to locally compact groupoids \cite{Ren_book}. This was followed by further studies, for example in  \cite{AD87, AD02} for group actions. A detailed general study is provided in the monograph \cite{AD-R}. In particular it has long been known \cite{Ren_book, Ren91, AD-R} that every amenable groupoid has the weak containment property, in the sense that its full and reduced $C^*$-algebras coincide. Yet, at that time the converse was left open: {\em is a locally compact groupoid amenable when it has the weak containment property}? For locally compact groups, this is well known to be true, due to a theorem of Hulanicki \cite{Hu}. More generally, this is true for any transitive locally compact groupoid \cite{bun} (that is, a groupoid acting transitively on its set of units). Later on it was proved \cite{Mat} that the weak containment property for an action of an exact discrete group on a compact space implies the amenability of the action, and recently the compactness assumption was removed \cite{BEW}. This leads  to believe that for any \'etale groupoid amenability and weak containment are two equivalent properties. However, in 2015 Willett  exhibited \cite{Wil15} a nice simple example of an \'etale groupoid having the weak containment property without being amenable. Willett's example is a bundle of groups. A related example, exhibiting this time a non-amenable principal groupoid having the weak containment property was afterwards found in \cite{AF-S}.

As explained in Section \ref{sect:WCamen} of this paper, and illustrated by several examples, it seems that, knowing {\it a priori} some exactness property of a locally compact groupoid, it is amenable as soon as it has the weak containment property. In a previous version of this paper we had stated that, for an inner exact \'etale groupoid, amenability was a consequence of the weak containment property. However, our proof is false and irreparable and we leave open the problem. Inner exactness is introduced in Definition \ref{def:reduction}. This notion is interesting and plays now a role in other contexts (see \cite{BL}, \cite{BEW}). 

In Section \ref{sect:semi},we turn to the case of discrete semigroups. We limit ourself to semigroups not too far from the case of discrete groups, namely inverse semigroups (defined in Section \ref{sec:sg}) and sub-semigroups of groups. We  address a recurrent question concerning semigroups: {\it what is the right definition of amenability for a semigroup}? To the semigroups  that we consider are attached a full $C^*$-algebra  and a reduced $C^*$-algebra, generalizing the classical case of groups. There are  three obvious candidates for the notion of amenability, and it is natural to wonder what are the relations  between them:
 \begin{itemize}
\item[(1)]  {\it left amenability}, that is, there exists a left invariant mean on the semigroup;
\item[(2)] {\it weak containment property}, that is, the full and reduced $C^*$-algebras of the semigroup are the same;
\item[(3)] {\it nuclearity} of the reduced $C^*$-algebra of the semigroup.
\end{itemize}
This problem has been considered in many papers (see \cite{Pat78}, \cite{DP}, \cite{Nica92}, \cite{LR96}, \cite{CL}, \cite{Mil}, \cite{Li12}, \cite{Li13}, \cite{ES}, to cite a few of them).  Of course, these three properties are equivalent for a discrete  group. 

Let us consider first the case of an inverse semigroup (see Section \ref{sec:sg}), that we denote by $S$. A very useful feature of  such a semigroup is that its full and reduced $C^*$-algebras are described via the groupoid $\cG_S$ canonically associated to it \cite{Pat}. As a consequence, we see that (3) $\Rightarrow$ (2) in this case:
{\it if the reduced $C^*$-algebra $C^*_{r}(S)$ is nuclear, then $S$ has the weak containment property}, since $\cG_S$ is amenable.

 It is the only general fact that can be stated.  {\it The example given by Willett}, once  reinterpreted in the setting of inverse semigroups, {\it allows us to show that (2) $\not\Rightarrow$ (3) in general, for inverse semigroups} (see Example \ref{ex:Willett}). This answers a question raised in \cite[Remark 3.7]{ES}. This example is a Clifford inverse semigroup, that is an inverse semigroup which is a disjoint union $S = \sqcup_{e\in E} S_e$ of groups where the set $E$ of idempotents is contained in the center of $S$. This gives an example of Clifford semigroup which has the weak containment property, although not all groups $S_e$, $e\in E$, are amenable. This answers a question raised in \cite{Pat78}.
 
 We observe that the notion of left amenability is not interesting, except when $S$ has not a zero element, {\it i.e.,} an element $0$ such that $0 s = 0 = s 0$ for every $s \in S$. Indeed,  any inverse semigroup with a zero is left amenable, since the Dirac measure at zero is a left invariant mean. Even {\it if $S$ has no zero, the left amenability of $S$ does not imply the weak containment property, and a fortiori the nuclearity of} $C^*_{r}(S)$ (see Example \ref{ex:Willett}).

Next, we consider the case of a pair $(P,G)$ where $P$ is a sub-semigroup of a group $G$ containing the unit $e$.   As pointed out in \cite{Li12,Li13,Nor14}, a handy tool in order to study the $C^*$-algebras of $P$ is its left inverse hull $S(P)$. It is an inverse semigroup with nice properties (Propositions \ref{prop:Eunit} and \ref{prop:Toeplitz}). Following Xin Li \cite{Li12,Li13}, we define the full $C^*$-algebra of $P$ to be the full $C^*$-algebra of $\cG_{S(P)}$. This extends the definition given by Nica in \cite{Nica92} for quasi-lattice ordered groups (Definition \ref{def:qlog}). On the other hand, $C^*_{r}(P)$ is a quotient of the reduced $C^*$-algebra of $\cG_{S(P)}$. 

We first observe that {\it the left amenability of $P$ always implies the nuclearity of} $C^*_{r}(P)$ (see Proposition \ref{prop:left_amen}). 
{\it It is not true in general that the weak containment property implies the left amenability of} $P$  as shown by Nica in \cite{Nica92}. He considered the free group $G=\F_n$ on $n$ generators $a_1,\dots,a_n$ and $P= \P_n$ is the semigroup generated by $a_1,\dots,a_n$. Using the uniqueness property of the Cuntz algebra $\cO_n$, Nica proved that $\P_n$ has the weak containment property although it is not left amenable. Note that $C^*_{r}(\P_n)$ is the Cuntz-Toeplitz $C^*$-algebra, that is, the $C^*$-algebra generated by $n$ isometries $s_1,\dots s_n$ such that $\sum_{1\leq i \leq n} s_i s_i^{*} \lneqq 1$. The weak containment property is equivalent to the uniqueness of the Cuntz-Toeplitz $C^*$-algebra. Moreover, $C^*_{r}(\P_n)$ is an extension of $\cO_n$ by the algebra of compact operators and therefore is nuclear.

In \cite{Li12}, Xin Li  introduced the independence property for $P$, which can be rephrased by saying that the quotient map from $C^*_{r}(\cG_{S(P)})$ onto  $C^*_{r}(P)$ is injective. In this case (which occurs for instance for  quasi-lattice ordered groups), the nuclearity of $C^*_{r}(P)$ implies the weak containment property for $\cG_{S(P)}$ and thus for $P$ (see Proposition \ref{prop:indep}).

 Whether the weak containment property for $P$ implies the nuclearity of $C^*_{r}(P)$  is an old problem that was raised by several authors, for instance by Laca and Raeburn \cite[Remark 6.9]{LR96}, and more recently by Xin Li \cite[\S 9]{Li13}. It is still open. We wonder  whether this is true when $G$ is exact.
 
  The sections 1 and 2 are devoted to preliminaries on locally compact groupoids and the notion of amenability.
 
We emphasize that {\em the locally compact spaces will always be Hausdorff } (unless explicitly mentioned) {\it and second countable. Locally compact groupoids  will  always come equipped with a Haar system and Hilbert spaces will be separable}.

\section{Preliminaries}

\subsection{Groupoids} We assume that the reader is familiar with the basic definitions about groupoids. For  details we refer to \cite{Ren_book}, \cite{Pat}. Let us recall some  notation and terminology.  A {\it groupoid} consists of a set $\cG$, a subset $\cG^{(0)}$ called the set of {\it units}, two maps $r,s :\cG \to \cG^{(0)}$ called respectively the {\it range} and {\it source} maps,  a {\it composition law} $(\gamma_1,\gamma_2) \in \cG^{(2)} \mapsto \gamma_1\gamma_2\in \cG$, where
$$\cG^{(2)} = \set{(\gamma_1,\gamma_2)\in \cG\times \cG : s(\gamma_1) = r(\gamma_2)},$$
and an {\it inverse} map $\gamma\mapsto \gamma^{-1}$. These operations satisfy obvious rules, such as the facts   that the composition law ({\it i.e.}, product) is associative, that the elements of $\cG^{(0)}$ act as units ({\it i.e.}, $r(\gamma)\gamma = \gamma = \gamma s(\gamma)$), that $\gamma\gamma^{-1} = r(\gamma)$, $\gamma^{-1}\gamma = s(\gamma)$, and so on (see \cite[Definition 1.1]{Ren_book}). For $x\in \cG^{(0)}$ we set $\cG^x = r^{-1}(x)$ and $\cG_x = s^{-1}(x)$. Usually, $X$ will denote the set of units of $\cG$.

A {\it locally compact groupoid} is a groupoid $\cG$ equipped with a  locally compact topology such  that the structure maps are continuous, where $\cG^{(2)}$ has the topology induced by $\cG\times\cG$ and $\cG^{(0)}$ has the topology induced by $\cG$. We assume that the range (and therefore the source) map is open, which is a necessary condition for the existence of a Haar system.  We denote by $\cC_c(\cG)$ the algebra of continuous complex valued functions with compact support on $\cG$.

 \begin{defn}\label{def:Haar} Let $\cG$ be a locally compact groupoid. A {\it Haar system} on $\cG$ is a family $ \lambda=(\lambda^x)_{x\in X}$ of measures on $\cG$, indexed by the set $X= \cG^{(0)}$ of units, satisfying the following conditions:
 \begin{itemize}
 \item {\it Support}: $\lambda^x$ has exactly $\cG^x$ as support, for every $x\in X$;
 \item {\it Continuity}: for every $f\in \cC_c(\cG)$, the function $x\mapsto  \lambda(f)(x)=\int_{\cG^{x} }f\rd\lambda^{x}$ is continuous;
\item {\it Invariance}: for $\gamma\in \cG$ and $f\in \cC_c(\cG)$, we have 
$$\int_{\cG^{s(\gamma)}} f(\gamma\gamma_1) \rd\lambda^{s(\gamma)}(\gamma_1) =  \int_{\cG^{r(\gamma)}} f(\gamma_1) \rd\lambda^{r(\gamma)}(\gamma_1).$$
\end{itemize}
\end{defn}

\begin{exs}\label{exs:groupoids} (a) {\it Transformation groupoid.} Let $G$ be a locally compact group acting conti\-nuously to the right on a locally compact space $X$. The topological product space $X\times G$ has a natural groupoid structure with $X$ as space of units. The range and source maps are given respectively by $r(x,g) = x$ and $s(x,g) = xg$. The product is given by $(x,g)(xg, h) = (x,gh)$ and the inverse by $(x,g)^{-1} = (xg,g^{-1})$. We denote by $X\rtimes G$ this groupoid. A Haar system $\lambda$ is given by $\lambda^{x} = \delta_x\times \tilde{\lambda}$ where $\tilde{\lambda}$ is a left Haar measure on $G$. Similarly, one defines $G\ltimes X$ for a left action of $G$.

(b) {\it Groupoid group bundle.} It is a locally compact groupoid such that the range and source maps are equal. By \cite[Lemma 1.3]{Ren91}, one can choose, for $x\in \cG^{(0)}$, a left Haar measure $\lambda^x$ on the group $\cG^x =\cG_x$ in such a way that $(\lambda^x)_{x\in X}$ forms a Haar system on $\cG$. An explicit example will be given in Section \ref{sec:reduce}.

(c) {\it Etale groupoids.} A locally compact groupoid is called {\it \'etale} when its range (and therefore its source) map is a local homeomorphism from $\cG$ onto $\cG^{(0)}$. Then $\cG^x$ and $\cG_x$ are discrete and $\cG^{(0)}$ is open in $\cG$. Moreover the family of counting measures $\lambda^x$ on $\cG^x$ forms a Haar system (see \cite[Proposition 2.8]{Ren_book}). Groupoids associated with actions, or more generally with partial actions (that we define now), of discrete groups are \'etale.

(d)   {\it Partial transformation groupoid.} A partial action of a discrete group $G$ on a locally compact space $X$ is a family 
$(\beta_g)_{g\in G}$  of partial homeomorphisms of $X$ between open 
subsets, such that $\beta_e = \Id_X$ and $\beta_g\beta_h \leq \beta_{gh}$ for $g,h\in G$, meaning that $\beta_{gh}$ extends $\beta_g\beta_h$. Then 
$$G\ltimes X = \set{(g,x) : g\in G, x\in X_{g^{-1}}}\subset G\times X$$
 with the  topology induced from the product topology, where $X_{g^{-1}}$ is the domain of $\beta_g$, is an \'etale groupoid. The range and source maps of $G\rtimes X$ are given respectively by $r(g,x) = \beta_g(x)$ and $s(g,x) = x$. The product is defined by $(g,x)(h,y) = (gh,y)$ when $x= \beta_h(y)$, and the inverse is given by $(g,x)^{-1} = (g^{-1}, \beta_g(x))$.
\end{exs}

 {\em In the sequel, the locally compact groupoids  are implicitly supposed to be Hausdorff, second countable, and equipped with a Haar system $\lambda$. In the three above examples, $\lambda$ will be the mentioned Haar system}.

\subsection{Representations of a locally compact groupoid.} Let $(\cG, \lambda)$ be a locally compact groupoid with a Haar system  $\lambda$. We set $X = \cG^{(0)}$. The space $\cC_c(\cG)$ is an involutive algebra with respect to the following operations for $f,g\in \cC_c(\cG)$:
\begin{align}
(f*g)(\gamma) &= \int f(\gamma_1)g(\gamma_{1}^{-1}\gamma) d\lambda^{r(\gamma)}(\gamma_1)\\
f^*(\gamma) & =\overline{f(\gamma^{-1})}. 
\end{align}

We define a norm on $\cC_c(\cG)$ by
$$\norm{f}_I = \max \set{\sup_{x\in X} \int \abs{f(\gamma)}\rd\lambda^x(\gamma), \,\,\sup_{x\in X} \int \abs{f(\gamma^{-1})}\rd\lambda^x(\gamma)}.$$

\begin{defn}\label{def:rep} A {\it representation}  of $\cC_c(\cG)$ is a $*$-homomorphism $\pi$ from $\cC_c(\cG)$ into the $C^*$-algebra $\cB(H)$ of bounded operators of a Hilbert space $H$ such that $\norm{\pi(f)} \leq \norm{f}_I$ for every $f \in \cC_c(\cG)$.
\end{defn}

\begin{ex}\label{ex:rep} Let $x\in X = \cG^{(0)}$. We denote by $\lambda_x$ the image of $\lambda^x$ by the inverse map $\gamma \mapsto \gamma^{-1}$. Let $\pi_x : \cC_c(\cG) \to \cB\big(L^2(\cG_x, \lambda_x)\big)$ be defined by
$$(\pi_x(f)\xi)(\gamma) = \int_{\cG_x} f(\gamma\gamma_1^{-1}) \xi(\gamma_1) \rd \lambda_x(\gamma_1)$$
for $f\in \cC_c(\cG)$ and $\xi\in L^2(\cG_x, \lambda_x)$. Then $\pi_x$ is a representation of $\cC_c(\cG)$.

More generally, let $\mu$ be a (Radon) measure on $X$. We denote by $\nu =\mu\circ\lambda$ the measure on $\cG$ defined by the formula 
$$\int_\cG f \rd\nu = \int_X \big(\int_{\cG^x} f(\gamma) \rd\lambda^x(\gamma)\big)\rd\mu(x).$$
Let $\nu^{-1}$ be the image of $\nu$ under the inverse map.  For $f\in \cC_c(\cG)$ and $\xi\in L^2(\cG,\nu^{-1})$ we define the operator 
$\hbox{Ind}_\mu(f)$ by the formula
$$\big(\hbox{Ind}_\mu(f)\xi\big)(\gamma) = \int _{\cG_{s(\gamma)}}  f(\gamma\gamma_1^{-1}) \xi(\gamma_1) \rd \lambda_{s(\gamma)}(\gamma_1).$$
Then $\hbox{Ind}_\mu$ is a representation of $\cC_c(\cG)$, called the {\it induced representation associated with} $\mu$.
We have $\hbox{Ind}_{\delta_x} = \pi_x$.
\end{ex}

The {\it full $C^*$-algebra} $C^*(\cG)$ of $\cG$ is the completion of $\cC_c(\cG)$ with respect to the norm
$$\norm{f} = \sup \norm{\pi(f)}$$
where $\pi$ runs over all representations  of $\cC_c(\cG)$. The {\it reduced $C^*$-algebra} $C_{r}^*(\cG)$ is the completion of $\cC_c(\cG)$ with respect to  the norm
$$\norm{f} = \sup_{x\in X} \norm{\pi_x(f)}.$$

Obviously, the identity map of $\cC_c(\cG)$ extends to a surjective homomorphism from $C^*(\cG)$ onto $C^*_{r}(\cG)$.

\begin{rem} Assume that $\cG$ is a locally compact group. Let us observe that the involution on $\cC_c(\cG)$ that we introduced is not the usual one in group theory. If $\Delta$ denotes the modular function of $\cG$, usually the involution is defined by $f^\star(\gamma) = \overline{f(\gamma^{-1})}\Delta(\gamma^{-1})$. The map $f\mapsto \tilde{f}$ where $\tilde{f}(\gamma) = f(\gamma) \Delta(\gamma)^{-1/2}$ is an isomorphism of involutive algebra between the $*$-algebra $\cC_c(\cG)$ with the involution $^*$ introduced in (2) and the usual one with the involution $^\star$. The full and reduced $C^*$-algebras defined above are then canonically identified respectively with the classical full and reduced $C^*$-algebras of the group $\cG$.

Similarly, the full and reduced $C^*$-algebras of a transformation groupoid $X\rtimes G$ are identified with the full crossed product $\cC_0(X) \rtimes G$ and the reduced crossed product $\cC_0(X) \rtimes_r G$ respectively, where $\cC_0(X)$ is the $C^*$-algebra of complex valued functions on $X$ vanishing to $0$ at infinity.
\end{rem}

A familiar result in group theory relates in a bijective and natural way    the non-degenerate representations of the full group $C^*$-algebra and the unitary representations of the group. A similar result holds for groupoids. Its statement requires some preparation.

Let $(\cG,\lambda)$ be a locally compact groupoid and let $\mu$ be a (Radon) measure on $X= \cG^{(0)}$. We set $\nu =\mu\circ\lambda$. We say that $\mu$ is {\it quasi-invariant} if $\nu$ is equivalent to $\nu^{-1}$. In this case, we denote by $\Delta$ the Radon-Nikod\'ym derivative $\rd\nu/\rd\nu^{-1}$. A groupoid $(\cG,\lambda)$ equipped with a quasi-invariant measure $\mu$ is called a {\it measured groupoid}.
 
 \begin{defn}\label{def:unitary_rep} A {\it unitary representation} of $\cG$ is a triple $(\mu,H, U)$ where
 \begin{itemize}
 \item[(i)] $\mu$ is a quasi-invariant measure on $X$;
 \item[(ii)] $H = (\set{H_x: x\in X}, \cE)$ is a measurable field of Hilbert spaces over $X$ (where $\cE$ is a fundamental sequence of measurable vector fields);
 \item[(iii)] $U$ is a measurable action of $\cG$ on $H$ by isometries, that is, for every $\gamma\in \cG$ we have an isometric isomorphism $U(\gamma) : H_{s(\gamma)} \to H_{r(\gamma)}$ such that
 \begin{itemize}
 \item[(a)] for $x\in X$, $U(x)$ is the identity map of $\cH_x$;
 \item[(b)] for $(\gamma,\gamma_1) \in \cG^{(2)}$, $U(\gamma\gamma_1) = U(\gamma)U(\gamma_1)$;
 \item[(c)] for $\xi,\eta\in \cE$, the function $\gamma\mapsto\scal{\xi\circ r(\gamma), U(\gamma)\eta\circ s(\gamma)}_{r(\gamma)}$ is measurable.
 \end{itemize}
 \end{itemize}
 \end{defn}
 
 We denote by $\cH = L^2(X, H,\mu)$ the Hilbert space of square integrable sections of $H$, and for $f\in \cC_c(\cG)$ we define the operator $\pi_U(f)$ on $\cH$ by the formula
 $$\scal{\xi, \pi_U(f)\eta} = \int_X\Big(\int_{\cG^x}f(\gamma)\Delta(\gamma)^{-1/2}\scal{\xi\circ r(\gamma), U(\gamma)\eta\circ s(\gamma)}_{r(\gamma)}\rd\lambda^x(\gamma)\Big) \rd\mu(x).$$
 Then $f\mapsto \pi_U(f)$ is a representation of $\cC_c(\cG)$, called the {\it integrated form of $(\mu,H, U)$} (or simply $U$). A crucial  result, due to J. Renault, asserts that every representation of $\cC_c(\cG)$ can be disintegrated. Here, the fact that the groupoid is assumed to be second countable is needed.
 
 \begin{thm}{\rm(\cite[Proposition 4.2]{Ren87})} Let $\pi$ be a non-degenerate representation of $\cC_c(\cG)$ on a Hilbert space $\cH$. There is a unitary representation $(\mu,H, U)$ of $\cG$ such that $\pi$ is unitary equivalent to the integrated form $\pi_U$ of $(\mu,H, U)$. We say that $\pi$ disintegrates over $\mu$.
 \end{thm}

 \begin{ex}\label{ex:regular} The {\it left regular representation of $\cG$ over a quasi-invariant measure} $\mu$ is $(\mu, H= L^2(\cG,\lambda), L)$ where $L^2(\cG,\lambda) = \big(\set{L^2(\cG^x,\lambda^x):x\in X}, \cE= \cC_c(\cG)\big)$, and 
 $$L(\gamma) : L^2(\cG^{s(\gamma)},\lambda^{s(\gamma)}) \to L^2(\cG^{r(\gamma)},\lambda^{r(\gamma)})$$ is given, for $\xi \in   L^2(\cG^{s(\gamma)},\lambda^{s(\gamma)})$, $\gamma_1\in \cG^{r(\gamma)}$, by
 $$\big(L(\gamma)\xi\big)(\gamma_1) = \xi(\gamma^{-1}\gamma_1).$$
 
 Note that $ L^2(X, H,\mu) = L^2(\cG, \mu\circ\lambda)$ and that, for $f\in \cC_c(\cG)$, $\xi \in L^2(\cG, \mu\circ\lambda)$
 we have 
 $$\big(\pi_L(f)\xi\big)_x = \int_{\cG^{x}} f(\gamma) \Delta(\gamma)^{-1/2} L(\gamma)\xi\circ s(\gamma) \rd \lambda^x(\gamma).$$
 
 It is well known (and easy to see) that the map $W: L^2(\cG,\nu) \to L^2(\cG,\nu^{-1})$ defined by the formula
 $W\xi = \Delta^{1/2}\xi$ is an isometric isomorphism which implements a unitary equivalence between $\pi_L$ and $\hbox{Ind}_\mu$.
 \end{ex}

 \subsection{About reductions of a groupoid}\label{sec:reduce} Let $\cG$ be a groupoid and $Y$ a subset of $X= \cG^{(0)}$. We set $\cG(Y) = r^{-1}(Y)\cap s^{-1}(Y)$. Then $\cG(Y)$ is a  subgroupoid of $\cG$ called the {\it reduction of $\cG$ by $Y$}.
 When $Y$ is reduced to a single element $x$, then $\cG(x) = r^{-1}(x) \cap s^{-1}(x)$ is a  group called the {\it isotropy group of $\cG$ at $x$}. 
 
 Let now $\cG$ be a  locally compact groupoid with Haar system and let $Y$ be a locally compact subset of $X$ which is $\cG$-invariant, meaning that for $\gamma\in \cG$, we have $r(\gamma) \in Y$ if and only if $s(\gamma)\in Y$.
 Then $\cG(Y)$ is a locally compact groupoid whose Haar system is obtained by restriction of the Haar system of $\cG$.
  
 Let us consider the general situation where  a closed $\cG$-invariant subset $F$ of $X$ is given and set $U = X\setminus F$. It is well known that the inclusion
$\iota_U :  \cC_c(\cG(U) )\to \cC_c(\cG)$ extends to  injective homomorphisms from $C^*(\cG(U))$ into $C^*(\cG)$ and from
$C^*_{r}(\cG(U))$ into $C^*_{r}(\cG)$. Similarly, the restriction map $p_F: \cC_c(\cG) \to \cC_c(\cG(F))$ extends to  surjective homomorphisms from $C^*(\cG)$ onto $C^*(\cG(F))$ and from
$C^*_{r}(\cG)$ onto $C^*_{r}(\cG(F))$. Moreover the sequence 
$$0 \rightarrow C^*(\cG(U)) \rightarrow C^*(\cG) \rightarrow C^*(\cG(F)) \rightarrow 0$$
is exact. For these facts, we refer to \cite[page 102]{Ren_book}, \cite[Section 2.4]{HS}, or to \cite[Proposition 2.4.2]{Ram} for a detailed proof. 

On the other hand, the corresponding sequence with respect to the reduced $C^*$-algebras is not always exact, as shown for a non-Hausdorff groupoid by Skandalis in the Appendix of \cite{Ren91}. 

\begin{ex}\label{ex:HLS} Another interesting class of examples was provided in \cite{HLS} by Higson, Lafforgue and Skandalis. There, the authors consider a residually finite group $\Gamma$ and an decreasing sequence $\Gamma \supset N_0 \supset N_1 \cdots \supset N_k \supset \cdots$ of finite index normal subgroups with $\cap_k N_k = \set{e}$. Let  $\widehat\N =\N\cup\set{\infty}$ be the Alexandroff compactification of $\N$. We set $N_\infty = \set{e}$ and, for $k\in  \widehat\N$, we denote by $q_k : \Gamma \to \Gamma/N_k$ the quotient homomorphism. Let $\cG$ be the quotient of $\widehat\N\times \Gamma$ with respect to the equivalence relation
$$(k,t)\sim (l,u) \,\,\,\hbox{if} \,\,\, k=l \,\,\,\hbox{and}\,\,\, q_k(t) = q_k(u).$$
Equipped with the quotient topology, $\cG$ has a natural structure of  (Hausdorff) \'etale locally compact groupoid group bundle: its space of units is $\widehat\N$, the range and source maps are given by $r([k,t]) = s([k,t]) = k$, where $[k,t] = (k,q_k(t))$ is the equivalence class of $(k,t)$. The fibre $\cG(k)$ of the bundle is the quotient group $\Gamma/N_k$ if $k\in \N$ and $\Gamma$ if $k=\infty$. We call this groupoid an HLS-{\it groupoid}.  A basic result of \cite{HLS} is that the sequence
$$0 \longrightarrow C^*_{r}(\cG(\N))\longrightarrow C^*_{r}(\cG) \longrightarrow C^*_{r}(\cG(\infty)) \longrightarrow 0$$
is not exact whenever  $\Gamma$ is infinite and has Kazdhan's property (T) (it is not even exact in $K$-theory!). In fact, we will see in Proposition \ref{prop:HLS} that this sequence is exact if and only if $\Gamma$ is amenable.
\end{ex}

 \subsection{Crossed products} For the definition of actions of groupoids on $C^*$-algebras we refer to \cite{KS04}. Let us recall a few facts.
 
 \begin{defn} Let $X$ be a locally compact space. A $\cC_0(X)$-{\it algebra} is a $C^*$-algebra $A$ equipped with a homomorphism $\rho$ from $\cC_0(X)$ into the centre of the multiplier algebra of $A$, which is non-degenerate  in the sense that there exists an approximate unit $(u_\lambda)$ of $\cC_0(X)$ such that $\lim_\lambda \rho(u_\lambda) a = a$ for every $a\in A$.
\end{defn}

Given $f\in \cC_0(X)$ and $a\in A$, for simplicity we will write $f a$ instead of $\rho(f) a$.

Let $U$ be an open subset of $X$ and $F= X\setminus U$. We view $\cC_0(U)$ as an ideal of $\cC_0(X)$ and we denote by $\cC_0(U) A$ the closed linear span of $\set{fa: f\in \cC_0(U), a\in A}$. It is a closed ideal of $A$ and in fact, we have $\cC_0(U) A = \set{fa: f\in \cC_0(U), a\in A}$ (see \cite[Corollaire 1.9]{Blan}). We set $A_F = A/\cC_0(U) A$ and whenever $F = \set{x}$ we write $\cC_x(X)$ instead of $\cC_0(X\setminus\set{x})$ and $A_x$ instead of $A_{\set{x}}$. We denote by $e_x : A \to A_x$ the quotient map and for $a\in A$ we set $a(x) = e_x(a)$. Recall that  $\norm{a} = \sup_{x\in X}\norm{a(x)}$ (so that $a \mapsto (a(x))_{x\in X}$ from $A$ into $\prod_{x\in X} A_x$ is injective) and that $x\mapsto \norm{a(x)}$ is upper semi-continuous (see \cite{Rie,Blan}). Then, $(A, \set{e_x : A \to A_x}_{ x\in X}, X)$ is an upper semi-continuous field of $C^*$-algebras.

 Let $A$ and $B$ be two $\cC_0(X)$-algebras. A {\it morphism $\alpha: A \to B$ of $\cC_0(X)$-algebras} is a morphism of $C^*$-algebras which is $\cC_0(X)$-linear, that is, $\alpha(fa) = f\alpha(a)$ for $f\in \cC_0(X)$ and $a\in A$. For $x\in X$, in this case $\alpha$ factors through a morphism $\alpha_x : A_x \to B_x$ such that $\alpha_x(a(x)) = \alpha(a)(x)$.
 
 Let $Y,X$ be locally compact spaces and $f :Y\to X$ a continuous map. To any $\cC_0(X)$-algebra $A$ is associated a $\cC_0(Y)$-algebra $f^*A = \big(\cC_0(Y)\otimes A\big)_F$ where $F= \set{(y,f(y)):y\in Y}\subset Y\times X$. For $y\in Y$, we have $\big(f^* A)_y = A_{f(y)}$ (see \cite{LeG, KS04}).

\begin{defn}(\cite{LeG, KS04}) Let $(\cG,\lambda)$ be a locally compact groupoid with a Haar system and $X = \cG^{(0)}$. An {\it action} of $\cG$ on a $C^*$-algebra $A$ is given by a structure of $\cC_0(X)$-algebra on $A$ and an isomorphism $\alpha : s^* A \to r^* A$ of $\cC_0(\cG)$-algebras such that for every $(\gamma_1,\gamma_2)\in \cG^{(2)}$ we have $\alpha_{\gamma_1\gamma_2} = \alpha_{\gamma_1}\alpha_{\gamma_2}$, where $\alpha_\gamma : A_{s(\gamma)} \to A_{r(\gamma)}$ is the isomorphism deduced from $\alpha$ by factorization.

When $A$ is equipped with such an action, we say that $A$ is a {\it $\cG$-$C^*$-algebra}.
\end{defn}

Let $A$ be a $\cG$-$C^*$-algebra. We set $\cC_c(r^*(A)) = \cC_c(\cG)r^*(A)$. It is the space of the continuous sections with compact support of the upper semi-continuous field of $C^*$-algebras defined by the $\cC_0(\cG)$-algebra $r^*A$. Then, $ \cC_c(\cG)r^*(A)$ is a $*$-algebra with respect to the following operations:
$$(f*g)(\gamma) = \int_{\cG^{r(\gamma)}} f(\gamma_1) \alpha_{\gamma_1}\big(g(\gamma_1^{-1}\gamma)\big) \rd\lambda^{r(\gamma)}(\gamma_1)$$
and
$$f^*(\gamma) = \alpha_\gamma\big(f(\gamma^{-1})^*\big)$$
(see \cite[Proposition 4.4]{MW}). We define a norm on $\cC_c(r^*(A))$ by
$$\norm{f}_I = \max \set{\sup_{x\in X} \int_{\cG^x}\norm{f(\gamma)}\rd\lambda^x(\gamma),\quad \sup_{x\in X} \int_{\cG^x}\norm{f(\gamma^{-1})}\rd\lambda^x(\gamma)}.$$
The {\it full crossed product} $A\rtimes_\alpha \cG$ is the enveloping $C^*$-algebra of the Banach $*$-algebra obtained by completion of $\cC_c(r^*(A))$ with respect to $\norm{\cdot}_I$.

Let us now define the reduced crossed product. For $x\in X$ let us consider the Hilbert $A_x$-module $L^2(\cG_x,\lambda_x)\otimes A_x$, defined as the completion of the space $\cC_c(\cG_x; A_x)$ of continuous compactly supported functions from $\cG_x$ into $A_x$, with respect to the $A_x$-valued inner product
$$\scal{\xi,\eta} = \int_{\cG_x} \xi(\gamma)^*\eta(\gamma) \rd\lambda_{x}(\gamma).$$
For $f\in \cC_c(r^*(A))$, $\xi\in \cC_c(\cG_x; A_x)$ and $\gamma\in \cG_x$, we set
$$\big(\pi_x(f)\xi\big)(\gamma) = \int_{\cG_x} \alpha_{\gamma}^{-1}\big(f(\gamma\gamma_{1}^{-1})\big)\xi(\gamma_1) \rd\lambda_x(\gamma_1).$$
Then $\pi_x(f)$ extends to a bounded operator with adjoint acting on the Hilbert $A_x$-module $L^2(\cG_x,\lambda_x)\otimes A_x$. In this way we get a representation of the $*$-algebra  $\cC_c(r^*(A))$. The reduced crossed product $A\rtimes_{\alpha,r} \cG$ is the completion of $\cC_c(r^*(A))$ with respect to the norm $\norm{f} = \sup_{x\in X}\norm{\pi_x(f)}$ (see \cite{KS04})\footnote{$\pi_x$ is what is denoted $\Lambda_x$ in \cite{KS04} except that the authors consider $\cC_c(s^*(A))$ instead of $\cC_c(r^*(A))$. This explains why our formula is not exactly the same.}.

\begin{rems}\label{rem:cross_P} (a) We note that if $Y$ is a locally compact $\cG$-invariant subset of $X = \cG^{(0)}$, then $\cC_0(Y)$ has a natural structure of $\cG$-$C^*$-algebra. Moreover, $\cC_c(r^*(\cC_0(Y))) = \cC_c(\cG(Y))$, and $\cC_0(Y) \rtimes \cG$ and  $\cC_0(Y) \rtimes_r \cG$ are canonically isomorphic to  $C^*(\cG(Y))$ and $C^*_{r}(\cG(Y))$ respectively.

(b) Let $B$ be a $C^*$-algebra and set $A = B \otimes\cC_0(X) $. Since $\cC_0(X)$ is a $\cG$-$C^*$-algebra, we see that  $A = B \otimes\cC_0(X) $ is a $\cG$-$C^*$-algebra, the action being trivial on $B$. Moreover,  $A\rtimes \cG$ and $A\rtimes_r \cG$ are canonically isomorphic to $B\otimes_{max}  C^*(\cG)$ and $B\otimes C^*_{r}(\cG)$ respectively.
\end{rems}

\section{Amenability and weak containment}

The reference for this section is \cite{AD-R}. The notion of amenable locally compact groupoid has many equivalent definitions. We will recall two of them. Before, let us recall a notation: given a locally compact groupoid $\cG$, $\gamma\in \cG$ and $\mu$ a measure on $\cG^{s(\gamma)}$, then $\gamma\mu$ is the measure on $\cG^{r(\gamma)}$ defined by 
$\int_{\cG^{r(\gamma)}} f \rd \gamma\mu = \int_{\cG^{s(\gamma)}} f(\gamma\gamma_1) \rd \mu(\gamma_1)$.

\begin{defn}\label{def:amen1}(\cite[Definitions 2.2.2, 2.2.8]{AD-R}) 
We say that $\cG$ is {\it amenable} if there exists a net $(m_i)$, where  $ m_i =(m_i^{x})_{x\in \cG^{(0)}}$ is a family of probability measures $m_i^{x}$ on $\cG^x$, such that
\begin{itemize}
\item[(i)] each $m_i$ is continuous in the sense that for all $f\in \cC_c(\cG)$, the function $x\mapsto \int f\rd m_i^{x}$ is continuous;
\item[(ii)] $\lim_{i} \norm{\gamma m_i^{s(\gamma)} - m_i^{r(\gamma)}}_1 = 0$ uniformly on the compact subsets of $\cG$.
\end{itemize}
\end{defn}
We say that $(m_i)_i$ is an {\it approximate invariant continuous mean} on $\cG$. Note that if $\cG$ is amenable and if $Y$ is a locally compact $\cG$-invariant subset of $X$, then the groupoid $\cG(Y)$ is amenable.

\begin{prop}\label{prop:amen3}{\rm(\cite[Proposition 2.2.13]{AD-R})} Let $(\cG,\lambda)$ be a locally compact groupoid with Haar system. Then $\cG$ is amenable if and only if there exists  a net $(g_i)$ of non-negative functions in $\cC_c(\cG)$ such that 
\begin{itemize}
\item[(a)] $\int g_i \rd\lambda^{x} \leq 1$ for every $x\in \cG^{(0)}$;
\item[(b)] $\lim_i \int g_i \rd\lambda^{x} = 1$ uniformly on the compact subsets of $\cG^{(0)}$;
\item[(c)] $\lim_i \int \abs{g_i(\gamma^{-1}\gamma_1) -  g_i(\gamma_1)}\rd\lambda^{r(\gamma)}(\gamma_1) = 0$ uniformly on the compact subsets of $ \cG$.
\end{itemize}
\end{prop}

We will also need the notion of measurewise amenability.

\begin{defn}\label{def:meas_amen}{\rm(\cite[Proposition 3.2.14]{AD-R})} Let $(\cG,\lambda)$ be a locally compact groupoid. 
\begin{itemize}
\item[(i)] Let $\mu$ be a quasi-invariant measure on $X$. We say that {\it the measured groupoid $(\cG,\lambda,\mu)$ is amenable} if there exists  a net $(g_i)$ of $(\mu\circ\lambda)$-measurable non-negative functions on $\cG$ such that
\begin{itemize}
\item[(a)] $\int g_i \rd\lambda^{x} = 1$ for  {\it a.e.} $x\in \cG^{(0)}$;
\item[(b)] $\lim_i \int \abs{g_i(\gamma^{-1}\gamma_1) -  g_i(\gamma_1)}\rd\lambda^{r(\gamma)}(\gamma_1) = 0$ in the weak*-topology of $L^\infty(\cG,\mu\circ\lambda)$.
\end{itemize}
\item[(ii)]We say that $(\cG,\lambda)$ is {\it measurewise amenable} if $(\cG,\lambda,\mu)$ is 
an amenable measured groupoid for every quasi-invariant measure $\mu$.
\end{itemize}
\end{defn}

\begin{rem}\label{rem:meas_amen} An amenable groupoid is measurewise amenable. The converse is true for a large family of groupoids, for instance all \'etale groupoids and all locally compact groups (see \cite[Theorem 3.3.7, Remark 3.3.9, Examples 3.3.10]{AD-R}). 
\end{rem}

\begin{thm}\label{thm:equiv_amen} Let $(\cG,\lambda)$ be a locally compact groupoid. Consider the following conditions:
\begin{itemize}
\item[(a)] $(\cG,\lambda)$ is measurewise amenable;
\item[(b)] for every $\cG$-$C^*$-algebra $A$, the canonical surjection from $A\rtimes \cG$ onto $A\rtimes_r \cG$ is injective; 
\item[(c)] $C^*_{r}(\cG)$ is nuclear;
\item[(d)] the canonical surjection from $C^*(\cG)$ onto $C^*_{r}(\cG)$ is injective.
\end{itemize}
Then, we have (a)  $\Rightarrow$ (b)  $\Rightarrow$ (c), (b) $\Rightarrow$ (d). Moreover, if the isotropy groups $\cG(x)$ are discrete for every $x\in X$, then (c) $\Rightarrow$ (a).
\end{thm}

 \begin{proof} For the proof of (a) $\Rightarrow$ (b) see \cite[Theorem 3.6]{Ren91} or \cite[Proposition 6.1.10]{AD-R}. That  (c) $\Rightarrow$ (a) when the isotropy is discrete  is contained in \cite[Corollary 6.2.14]{AD-R}. To show that (b) $\Rightarrow$ (c), we take a $C^*$-algebra $A$ and the trivial action of $\cG$ on $\cC_0(X)\otimes A$ (where $X = \cG^{(0)}$).  Assuming that (c) holds true we have
$$A\otimes C^*_{r}(\cG) = (A\otimes\cC_0(X))\rtimes_r \cG =  (A\otimes \cC_0(X))\rtimes \cG = A\otimes_{max} C^*(\cG) = A\otimes_{max} C^*_{r}(\cG).$$
\end{proof}

\begin{defn}\label{def:WC} We say that a locally compact groupoid $(\cG,\lambda)$ has the {\it weak containment property} if Condition (d) of the previous theorem is fulfilled.
\end{defn}

Every amenable locally compact groupoid has the weak containment property. We study the converse in the next section.
\section{Weak containement vs amenability}\label{sect:WCamen}

\subsection{Some results} We first sum up the known results where weak containment implies amenability.

\begin{prop}$($\cite{bun}$)$ Let $(\cG,\lambda)$ be a transitive locally compact groupoid with Haar system having the weak containment property. Then $\cG$ is amenable.
\end{prop}

This fact can also been shown using the notion of equivalence of groupoids. Recall that a groupoid $\cG$ is transitive if for every $x,y\in X = \cG^{(0)}$ there exists $\gamma \in \cG$ with $r(\gamma) = x$ and $s(\gamma) = y$. Given $x\in X$, the groupoids $\cG$ and $\cG(x)$ are equivalent (see \cite{MRW}). Moreover the weak containment property and amenability are preserved under equivalence (see \cite{SW} for the first property and \cite{AD-R} for the latter one). Then we apply Hulanicki's result for groups.

\begin{thm}$($\cite{Mat}, \cite{BEW}$)$ Let $G\actson X$ be an action of a {\it \bf discrete exact} group on a locally compact group. Assume that the transformation groupoid $G\ltimes X$ has the weak containment property. Then $G\ltimes X$ is amenable
\end{thm}

\begin{ex}Another example where the weak containment property implies amenability concerns the natural action of a discrete group $G$ on its boundary $\partial G= \beta G\setminus G$.  Indeed, from the commutativity of the diagram
$$\xymatrix{
0\ar[r] &C^*(G\ltimes G) \ar@{=}[d] \ar[r] & C^*(G\ltimes \beta G)\ar[d] \ar[r] &C^* (G\ltimes \partial G)
\ar[r]\ar[d]^{\simeq} &0\\
0\ar[r] & C^*_{r}(G\ltimes G) \ar[r] & C^*_{r}(G\ltimes \beta G)\ar[r]&  C^*_r (G\ltimes \partial G)
\ar[r] &0}$$
since the first line is an exact sequence the second line is exact also. Roe and Willett proved in \cite{RoeW} that this exactness property implies that $G$ has Yu's property A and thus is exact. Therefore the action of $G$ on $\partial G$ is amenable by the previous theorem.
\end{ex}

\noindent{\bf Problem.} If $G\actson \beta G$ has the weak containment property, is it true that the action is amenable and thus that $G$ is exact? (see \cite[page 12]{BEW18}).

\subsection{The case of bundles of groups}

Let $\cG$ be a  groupoid group bundle over $X = \cG^{(0)}$. For $x\in X$, we rather denote by $G(x)$ the group $\cG(x)$.
 We set $U_x= X\setminus\set{x}$ and denote by $\cG(U_x)$ the subgroupoid of those $\gamma\in \cG$ such that $r(\gamma) \in U_x$. Let $\pi_x$ be the canonical surjective map from $C^*_{r}(\cG)$ onto $C^*_{r}(G(x))$.
\vspace{0.2cm}

\begin{prop}\label{prop:WC-amen}  Let $\cG$ be a groupoid group bundle and let us consider  the following conditions:\begin{itemize} 
\item[(1)] $\cG$ is amenable;
\item[(2)] for every $a\in C^*_{r}(\cG)$ the function $x \mapsto \norm{\pi_x(a)}$ is continuous;
\item[(3)] the sequence $0\to C^*_{r}(\cG(U_x)) \to  C^*_{r}(\cG) \to C^*_r (G(x)) \to 0$ is exact for every $x\in X$.
\end{itemize}
Then we have (1) $\Rightarrow$ (2) $\Leftrightarrow$ (3). Moreover these three conditions are equivalent when $\cG$ has the weak containment property.
\end{prop}

\begin{proof}
 (1) $\Rightarrow$ (3). Assume that $\cG$ is amenable. Then we have $C^*_{r}(\cG) = C^*(\cG)$. Moreover the groupoid $\cG(U_x)$ and the group $G(x)$ are also amenable and therefore their  reduced and full $C^*$-algebras coincide. It follows that the sequence in  (3) is exact. 
 
 Let us prove the equivalence between (2) and (3).  We know that 
 $$(C^*_{r}(\cG), \set{\pi_x : C^*_{r}(\cG) \to C^*_r (G(x))}_{x\in X}, X)$$
  is a field of $C^*$-algebras on $X$, which is lower semi-continuous in the sense that $x\mapsto \norm{\pi_x(a)}$ is  lower semi-continuous for every $a\in C^*_{r}(\cG)$ (see for instance \cite[Theorem 5.5]{LR}). On the other hand, $C^*_{r}(\cG)$ is a $\cC_0(X)$-algebra. Indeed, for $f\in \cC_0(X)$,  $g \in \cC_c(\cG)$  and $\gamma\in \cG$, we set $(fg)(\gamma) = f\circ r(\gamma) g(\gamma)$. The map $g\mapsto fg$ extends continuously in order to define a structure of $\cC_0(X)$-algebra on $C^*_{r}(\cG)$. We have $\cC_c(\cG(U_x)) = \cC_x(X)\cC_c(\cG)$ and by continuity we get $C^*_{r}(\cG(U_x)) = \cC_{x}(X)C^*_{r}(\cG)$. Note that for $f\in \cC_0(X)$, $a\in C^*_{r}(\cG)$ and $x\in X$, we have $\pi_x(fa) = f(x) \pi_x(a)$. It follows from \cite[Lemma 2.3]{KW95} that the function $x\mapsto \norm{\pi_x(a)}$ is upper semi-continuous at $x_0$ for all $a\in C^*_{r}(\cG)$ if and only if the kernel of $\pi_{x_0}$ is $C^*_{r}(\cG(U_{x_0}))$.
 
 Assume now that $\cG$ has the weak containment property.
 For every $x\in X$ the following diagram is commutative
 $$\xymatrix{
0\ar[r] &C^*(\cG(U_x)) \ar[d] \ar[r] & C^*(\cG)\ar[d]^{\lambda} \ar[r] &C^* (G(x))
\ar[r]\ar[d]^{\lambda_x} &0\\
0\ar[r] & C^*_{r}(\cG(U_x)) \ar[r] & C^*_{r}(\cG)\ar[r]^{\pi_x}&  C^*_r (G(x))
\ar[r] &0}$$
where the first line is exact and $\lambda$ is an isomorphism. By a straightforward diagram chasing we see  that $\lambda_x$ is injective ({\it i.e.}, $G(x)$ is amenable) if and only if the second line is exact.  Moreover the groupoid group bundle $\cG$ is amenable if and only if $G(x)$ is amenable for every $x\in X$ (see \cite[Example 5.1.3. (1)]{AD-R}).
\end{proof}

 Thus the natural question is whether the condition (3) of the previous proposition always holds when $\cG$ has the weak containment property. The  nice example \ref{ex:Wil} below  of Willett \cite{Wil15} shows that the answer is no. It is a well chosen  HLS groupoid (see Example \ref{ex:HLS}).

Let us keep the  notation of  Example \ref{ex:HLS}. In the case of HLS groupoids note that we have the following result.

\begin{prop}\label{prop:HLS}  We assume that $\Gamma$ is finitely generated. Then the following conditions are equivalent:
\begin{itemize}
\item[(1)] $\Gamma$ is amenable;
\item[(2)] $\cG$ is amenable;
\item[(3)]  $C^*_{r}(\cG)$ is a continuous field of  $C^*$-algebras with fibres $C^*_{r}(\cG(x))$;
\item[(4)] the sequence
$0 \longrightarrow C^*_{r}(\cG(\N))\longrightarrow C^*_{r}(\cG) \longrightarrow C^*_{r}(\cG(\infty)) \longrightarrow 0$
is exact.
\end{itemize}
\end{prop}

\begin{proof} Note that the fibres $\cG(k)$ of the groupoid group bundle $\cG$ are finite groups if $k\in \N$, and that $\cG(\infty) = \Gamma$. Therefore (1) $\Leftrightarrow$ (2) $\Rightarrow$ (3) $\Leftrightarrow$ (4) follows from the proposition \ref{prop:WC-amen}. That (3) $\Rightarrow$ (1) is proved in \cite[Proposition 9.14]{AD}.
\end{proof}
 
\begin{ex}\label{ex:Wil} There exist HLS groupoids that have the weak containment property and are not amenable, that is with $\Gamma$ non-amenable. Indeed, Willett takes $\Gamma$ to be the free group with two generators, and constructs a decreasing sequence $\Gamma \supset N_0 \supset N_1 \cdots \supset N_k \supset \cdots$ of finite index normal subgroups of $\Gamma$ with $\cap_k N_k = \set{e}$ such the correspondinding HLS groupoid has the weak containment property.
\end{ex}

\subsection{About the role of exactness} We see in the above examples that when knowing {\it a priori} some exactness property of the \'etale groupoid it is possible to show that its weak containment property implies the amenability of the groupoid. We would like to understand to what extent this is a more general fact. Let us give first some definitions.

\begin{defn}\label{def:reduction}We say that an action of  a locally compact groupoid $(\cG,\lambda)$ on a $C^*$-algebra $A$ is {\it inner exact} if for every $\cG$-invariant closed ideal $I$ of $A$ the sequence 
$$0\longrightarrow I\rtimes_r \cG \longrightarrow  A\rtimes_r \cG\longrightarrow   (A/I)\rtimes_r \cG \longrightarrow 0$$
is exact.
 We say that $\cG$ is {\it inner exact}  if the canonical action of $\cG$ on $\cC_0(\cG^{(0)})$ is inner exact, {\it i.e.,} if for every  invariant closed subset $F$ of $\cG^{(0)}$, the  sequence 
$$0\longrightarrow C^*_{r}(\cG(U)) \longrightarrow C^*_{r}(\cG)\longrightarrow  C^*_{r} (\cG(F))\longrightarrow 0$$
is exact, where $U = X\setminus F$.
\end{defn}

The term ``inner'' in the above definitions aims to highlight that we only consider short sequences with respect to the specific given action of the groupoid. A possible definition of exactness for a groupoid is the following one. Other candidates are considered in \cite{AD}.

\begin{defn}\label{def:KWexact} We say that a groupoid $(\cG,\lambda)$ is {\it exact in the sense of Kirchberg-Wassermann} (or KW-{\it exact} in short) if  every action of $(\cG,\lambda)$ on any $C^*$-algebra $A$ is inner exact.
\end{defn}

This notion was studied by Kirchberg and Wassermann for locally compact groups. They proved in particular that this property is equivalent to the exactness of $C^*_{r}(G)$ for a discrete group $G$. 

\begin{exs}\label{ex:inner_exact} (a) Amenable locally compact groupoids are KW-exact since for them the reduced and full crossed products relative to any action are the same.

(b) Every minimal groupoid is inner exact. In particular, every locally compact group is inner exact.

(c) Every KW-exact groupoid is inner exact.

(d) Let $G$ be a locally compact KW-exact group acting to the right on a locally compact space $X$. Then the transformation groupoid $\cG= X\rtimes G$ is KW-exact. Indeed let $\alpha$ be an action of $\cG$ on a $\cC_0(X)$-algebra $A$. Then $G$ acts on $A$ by $(\beta_g a)(x)  = \alpha_{(x,g)}(a(xg))$ and it is straightforward to check that
$A\rtimes_{\beta,r} G$ is canonically isomorphic to $A\rtimes_{\alpha,r} \cG$. Moreover, this identification is functorial.

In fact the groupoid is exact in a very strong sense. Indeed, $G$ acts amenably on a compact space $Y$ \footnote{In the non discrete case this fact is proved in a recent preprint of Brodzki, Cave and Kang Li \cite{BCL}.} and therefore it acts amenably on $Y\times X$ by $(y,x)g = (yg,xg)$ (see \cite[Proposition 2.2.9]{AD-R}). Then $\cG =X\rtimes G$ acts amenably on $Y\times X$ by $(y,x)(x,g) = (yg,xg)$ and the momentum map $(y,x) \mapsto x\in \cG^{(0)}$ is proper\footnote{This property holds true more generally for every partial transformation groupoid relative to a partial action of an exact group (see \cite[Proposition 4.18]{AD}).}. These facts imply that $\cG$ is KW-exact. The proof is the same as the proof  showing that a group acting amenably on a compact space is KW-exact (see for instance \cite[Theorem 7.2]{AD02}).More details on the notion of exactness for groupoids are given in \cite{AD}.

\end{exs}

\begin{rem}\label{rem:main} Let $(\cG,\lambda)$ be  an amenable locally compact groupoid. Then  $C^*(\cG) = C^*_{r}(\cG)$ and $\cG$ is inner exact. In the first version of this paper, we had stated that the converse is true, but our proof has a fatal error. However, we still believe that the result holds under some appropriate notion of exactness.
\end{rem}

\section{Semigroup $C^*$-algebras}\label{sect:semi}

\subsection{Semigroups}\label{sec:sg} We will consider two kinds of semigroups: inverse semigroups and sub-semi\-groups of a group.

An {\it inverse semigroup} $S$ is a semigroup such that for every $s\in S$ there exists a unique element $s^*$ such that $ss^*s = s$ and $s^*ss^* = s^*$. Our references for this notion are \cite{Pat,Law}. Note that groups are inverse semigroups with exactly one idempotent. The set $E_S$ of idempotents of $S$ plays a crucial role. It is an abelian sub-semigroup of $S$.  On $S$ one defines the equivalence relation $s\underset{\sigma}{\sim} t$ if there exists an idempotent $e$ such that $se = te$.
The quotient $S/\sigma$ is a group, called the {\it maximal group homomorphism image} of $S$, since every homomorphism from $S$ into a group $G$ factors through $S/\sigma$. This group $S/\sigma$ is trivial when $S$ has a zero element $0$, which is a frequent situation.

 By an abuse of notation,  $\sigma$ will also denote the quotient map from $S$ onto $S/\sigma$. If $S$ has a zero, we denote by $S^\times$ the set $S\setminus\set{0}$. When $S$ does not have a zero, we set $S^\times = S$.
 
Given a set $X$, we denote by $\IS(X)$ the inverse semigroup of partial bijections of $X$. Its zero element $0$ is the application with empty domain. The Wagner-Preston theorem \cite[Proposition 2.1.3]{Pat} identifies any inverse semigroup $S$ with a sub-semigroup of $\IS(S)$.

Let $G$ be a group and $P$ a sub-semigroup of $G$ containing the unit $e$. The {\it left inverse hull} $S(P)$ of $P$ is the inverse sub-semigroup of $\IS(P)$ generated by the injection $\ell_p : x\mapsto px$. It has a unit, namely $\ell_e$. Observe that $\ell_p^{*}$ is the map $px \mapsto x$ defined on $pP$. Every  element of $S(P)$ is of the form $s= \ell_{p_1}^*\ell_{q_1}\cdots \ell_{p_n}^*\ell_{q_n}$ with $p_i,q_i\in P$ and $n\geq 1$. Let us recall some important properties of $S(P)$. 

\begin{prop}\label{prop:Eunit} Let $(P,G)$ as above. Then
\begin{itemize}
\item[(1)] $0\not\in S(P)$ if and only if $PP^{-1}$ is a subgroup of $G$;
\item[(2)]  The function $\psi : S(P)^\times \to G$ sending $s=  \ell_{p_1}^*\ell_{q_1}\cdots \ell_{p_n}^*\ell_{q_n}$ to $p_{1}^{-1}q_1 \cdots p_{n}^{-1}q_n$ is well defined. It satisfies $\psi(st) = \psi(s)\psi(t)$ if $st\not=0$ and we have $\psi^{-1}(e) = E_{S(P)}^\times$.
\end{itemize}
\end{prop}

\begin{proof} $PP^{-1}$ is a subgroup of $G$ if and only if $pP\cap qP \not =\emptyset$ for every $p,q\in P$ ({\it i.e.,} $P$ is {\it left reversible}). Then, assertion (1) is Lemma 3.4.1 of \cite{Nor14}.

(2) is proved in \cite[Proposition 3.2.11]{Nor14}.
\end{proof}

Recall that on an inverse semigroup $S$, a partial order is defined as follows: $s\leq t$ if there exists an idempotent $e$ such that $s = te$ (see \cite[page 21]{Law} for instance).

\begin{defn}\label{def:idpure1} An inverse semigroup $S$ is said to be {\it $E$-unitary} if $E_S$ is the kernel of $\sigma : S\to S/\sigma$ (equivalently, every element greater than an idempotent is an idempotent). When $S$ has a zero, this means that $S= E_S$.
\end{defn}

\begin{defn}\label{def:idpure} Let $S$ be an inverse semigroup. A {\it morphism} (or {\it grading}) is a function $\psi$ from $S^\times$ into a group $G$ such that $\psi(st) = \psi(s)\psi(t)$ if $st\not=0$. If in addition  $\psi^{-1}(e) = E_{S}^\times$, we say that $\psi$ is  an {\it idempotent pure morphism}.  When such a function $\psi$ from $S^\times$ into a group $G$ exists, the inverse semigroup $S$ is called {\it strongly $E^*$-unitary}. 
\end{defn}

Note that when $S$ is without zero, $S$ is strongly $E^*$-unitary if and only if it is $E$-unitary.

\begin{rem}\label{rem:Eunit} Let $(P,G)$ with $G = PP^{-1}$. Then the map $\tau : S(P)/\sigma \to G$ such that $\tau\circ \sigma = \psi$ is an isomorphism. Indeed $\psi$ is surjective, so $\tau$ is also surjective. Assume that $\tau(\sigma(x)) = e$, with $x\in S(P)$. Since $\psi$ is idempotent pure, we see that $x$ is an idempotent and therefore $\sigma(x)$ is the unit of $S(P)/\sigma$.
\end{rem}

In \cite{Nica94}, Nica has introduced the Toeplitz inverse semigroup $S(G,P)$ which is the inverse sub-semigroup of $\IS(P)$ generated by the maps $\alpha_g : g^{-1}P \cap P \to P \cap gP$, $g\in G$, where $\alpha_g(x) = gx$ if $x \in g^{-1}P \cap P$.  For $p\in P$ we have $\alpha_p =\ell_p$. Therefore we have $S(P)\subset S(G,P)$. 

\begin{defn}\label{def:Toeplitz}We say that $(P,G)$ satisfies the {\it Toeplitz condition} if $S(P) = S(G,P)$.
\end{defn}

We will give in Section \ref{sect:alg_sg} another characterization of the Toeplitz condition, along with examples.

\begin{prop}\label{prop:Toeplitz} Assume that $(P,G)$ satisfies the Toeplitz condition. Let $\psi : S(P)^\times \to G$ as defined in Proposition \ref{prop:Eunit}. Then, $\alpha_g \not= 0$ if and only if $g$ is in the image of $\psi$. In this case we have  $\psi(\alpha_g) = g$, and $\alpha_g$ is the greatest element of $\psi^{-1}(g)$.
\end{prop} 

For the proof, see \cite[Proposition 4.1]{Nor15} or \cite[Lemma 3.2]{Nica94}.

\subsection{Groupoid associated with an inverse semigroup} 

Let $S$ be an inverse semigroup. We recall the construction of the associated groupoid $\cG_S$ that is described in detail in \cite{Pat}. We denote by $X$ the space of non-zero maps $\chi$ from $E_S$ into $\set{0,1}$ such that $\chi(ef) =\chi(e)\chi(f)$ and $\chi(0) = 0$ whenever $S$ has a zero.  Equipped with the  topology induced from the product space 
$\set{0,1}^E$, the space  $X$,  called the {\it spectrum} of $S$, is  locally compact and  totally disconnected. Note that when $S$ is a monoid ({\it i.e.,} has a unit element $1$) then $\chi$ is non-zero if and only if $\chi(1) = 1$, and therefore $X$ is compact.  

The semigroup $S$ acts on $X$ as follows. The domain (open and compact) of $t\in S$ is $D_{t^*t} = \set{\chi\in X: \chi(t^*t) = 1}$ and we set $\theta_t(\chi)(e) = \chi(t^*et)$.  We define on $\Xi = \set{(t,\chi)\in S\times X: \chi\in D_{t^*t}}$ the equivalence relation
$(t,\chi)\sim (t_1,\chi_1)$ if $\chi = \chi_1$ and there exists $e\in E_S$ with $\chi(e) = 1$ and $te= t_1e$.
Then $\cG_S$ is the quotient of $\Xi$ with respect to this equivalence relation, equipped with the quotient topology. The range of the class $[t,\chi]$ of $(t,\chi)$ is $\theta_t(\chi)$ and its source is $\chi$. The composition law is given by $[u,\chi][v, \chi'] = [uv,\chi']$ if $\theta_v(\chi')= \chi$ (see \cite{Pat} or \cite{Exel} for details). In general, $\cG_S$ is not Hausdorff. But for the inverse semigroups we are interested in, like $S(P)$, we will see that the quotient topology is Hausdorff.

\begin{prop}\label{prop:strongEunit} Let $S$ be a strongly $E^*$-unitary inverse semigroup, and let $\psi : S^\times \to G$ be an idempotent pure morphism. Then there is a partial action of $G$ on the spectrum $X$ of $S$ such that the groupoid $\cG_S$ is topologically isomorphic to the groupoid $G\ltimes X$ associated with the partial action. In particular, $\cG_S$ is Hausdorff and \'etale. Moreover, $X$ is compact when $S$ has a unit.
\end{prop}

\begin{proof} This result is described in \cite{MS}.  The partial action of $G$ on the spectrum $X$ of $S$ is defined by setting
$X_{g^{-1}} = \cup_{t\in \psi^{-1}(g)} D_{t^*t}$ (which can be empty). For $\chi \in X_{g^{-1}} $ we set $\beta_g(\chi) = \theta_t(\chi)$ where $t \in \psi^{-1}(g)$ is such that $\chi \in D_{t^*t}$.  This does not depend on the choice of $t$ as shown in \cite[Lemma 3.1]{MS}. Moreover, by \cite[Theorem 3.2]{MS}, the groupoid $\cG_S$ is canonically isomorphic to $G\ltimes X$.\footnote{In \cite{MS}, the proofs are carried out assuming that $S$ is $E$-unitary ({\it i.e.,} without zero) but they immediately extend to our setting.}
\end{proof}

\begin{prop}\label{prop:equiv_groupoid}  Let $S$ be a strongly $E^*$-unitary inverse semigroup. We assume that there is  an idempotent pure morphism $\psi : S^\times \to G$ such that if $g\not = e$ is in the image of $\psi$, then $\psi^{-1}(g)$ has a greatest element $\alpha_g$.
Then the groupoid $\cG_S$ is  equivalent \footnote{For this notion of equivalence of groupoids we refer  to \cite[Definition 2.2.15]{AD-R}.}
 to a transformation groupoid $G\ltimes Y$ for an action of $G$ on an Hausdorff locally compact space $Y$.
\end{prop}

\begin{proof} Let $t\in S^\times$ be such that $\psi(t) = g$. Since $t\leq \alpha_g$ we have $D_{t^*t} \subset D_{\alpha_{g}^*\alpha_g}$ and therefore $X_{g^{-1}} = D_{\alpha_{g}^*\alpha_g}$ is a closed subset of $X$. Moreover, for $\chi \in X_{g^{-1}}$ we have $\beta_g(\chi) = \theta_{\alpha_g}(\chi)$.  It follows that the cocycle $c: G\ltimes X \to G$ sending $(g,x)$ to $g$ is injective and closed. Injectivity means that the map $\gamma\in G\ltimes X  \mapsto (c(\gamma), s(\gamma))$ is injective. The cocycle is said to be closed if $\gamma \mapsto (r(\gamma), c(\gamma), s(\gamma))$ from $G\ltimes X$ into $X\times G\times X$ is closed. Since the cocycle $c$ is injective and closed, there exists a locally compact space $Y$ endowed with a continuous action of $G$ such that the transformation groupoid $G\ltimes Y$ is equivalent  to $G\ltimes X = \cG_S$. When $S$ has a unit, $X$ is compact and the equivalence is given by a groupoid isomorphism   $j$ from $G\ltimes X$ onto a reduction of $G\ltimes Y)$ (see \cite[Theorem 1.8]{KS02} and \cite[Theorem 6.2]{RS}).
\end{proof}

\begin{cor}\label{cor:toeplitz} Let  $P$ be a sub-semigroup of a group $G$ containing the unit.
\begin{itemize}
\item[(i)] The groupoid $\cG_{S(P)}$ is defined by a partial action of $G$ on a compact space.
\item[(ii)] If $(P,G)$ satisfies the Toeplitz condition, then $\cG_{S(P)}$ is equivalent to a transformation groupoid defined by an action of $G$ on a locally compact space.
\end{itemize}
\end{cor}

\begin{proof} (i) By Proposition \ref{prop:Eunit} there is an idempotent pure morphism $\psi:S(P)^\times \to G$ and we use Proposition \ref{prop:strongEunit}. 

(ii) follows from Propositions \ref{prop:Toeplitz}  and \ref{prop:equiv_groupoid}.
\end{proof}

\subsection{Weak containment for inverse semigroups}
Let $S$ be an inverse semigroup. Let us recall the definition of the full and reduced $C^*$-algebras of $S$ (for more details, see \cite[\S 2.1]{Pat}). Given $f,g\in \ell^1(S)$, we set
$$(f\star g)(t) = \sum_{uv = t} f(u) g(v), \quad f^*(t) = \overline{f(t^*)}.$$
Then $\ell^1(S)$ is Banach $*$-algebra, and the {\it full $C^*$-algebra} $C^*(S)$ of $S$ is defined as the enveloping $C^*$-algebra of $\ell^1(S)$. It is the universal $C^*$-algebra for the representations of $S$ by partial isometries. The {\it left regular representation} $\pi_2: S \to \cB(\ell^2(S))$ is defined by
$$\pi_2(t)\delta_u = \delta_{tu} \,\,\,\hbox{if}\,\,\, (t^*t)u =u,\quad \pi_2(t)\delta_u = 0 \,\,\, \hbox{otherwise}.$$
The extension of $\pi_2$ to $\ell^1(S)$ is faithful. The {\it reduced $C^*$-algebra} $C^*_{r}(S)$ of $S$ is the sub-$C^*$-algebra of $\cB(\ell^2(S))$ generated by $\pi_2(S)$. We still denote by $\pi_2: C^*(S) \to C^*_{r}(S)$ the extension of the left regular representation to $C^*(S)$.

When $S$ has a zero, we have $\pi_2(0)\delta_0 = \delta_0$ and $\pi_2(0)\delta_t = 0$ if $t\not= 0$. It follows that $\C\delta_0$ is an ideal in $C^*(S)$ that it is preferable to get rid of. So we set $C^*_{0}(S) = C^*(S)/\C\delta_0$ and similarly
 $C^*_{r,0}(S) = C^*_{r}(S)/\pi_2(\C\delta_0)$. We denote by $\pi_{2,0}$ the canonical surjective homomorphism from $C^*_{0}(S)$ onto $C^*_{r,0}(S)$ (see \cite{Nor14}).
 
 As shown in \cite{Pat} and \cite{KS02}, the $C^*$-algebras $C^*_{0}(S)$ and $C^*_{r,0}(S)$ are canonically isomorphic to 
$C^*(\cG_S)$ and $C^*_{r}(\cG_S)$ respectively.\footnote{ More precisely in \cite{Pat,KS02}, the authors consider the $C^*$-algebras $C^*(S)$ and $C^*_{r}(S)$, but their definition of $\cG_S$ is also slightly different because for the space $X = \cG_{S}^{(0)}$ they do not require that the map $\chi$ from $E_S$ into $\set{0,1}$ satisfy $\chi(0) = 0$. Their proof also works in our setting.} This is the reason for having introduced $C^*_{0}(S)$ and $C^*_{r,0}(S)$. Note that $C^*_{r,0}(S)$ is nuclear if and only $C^*_{r}(S)$ is so, and that $\pi_2$ is injective if and only if it is the case for $\pi_{2,0}$.

 \begin{defn} We say that $S$ has the {\it weak containment property} if $\pi_2$ (or equivalently $\pi_{2,0}$) is an isomorphism.
 \end{defn}

Observe that $S$ has the weak containment property if and only if the groupoid $\cG_S$ has this property.

Recall that a semigroup $S$ is {\it left amenable} if there exists a left invariant mean on $\ell^\infty(S)$. An inverse semigroup with zero is of course left amenable since the Dirac measure at zero is a left invariant mean. 

 If $C^*_{r}(S)$ is nuclear, the groupoid $\cG_S$ is amenable and therefore $S$ has the weak containment property. What about the converse?

 The following example shows that left amenability, even in the absence of zero, does not imply  the weak containment property. It also shows that the weak containment property is strictly weaker than  the nuclearity of $C^*_{r}(S)$.
 
 \begin{ex}\label{ex:Willett} Let $\Gamma$ be a residually finite group and $(N_k)_{k\geq 0}$ a decreasing sequence as in Example \ref{ex:HLS}, whose notation we keep. Let $S = \set{q_k(t) : k\in \widehat{\N}, t\in \Gamma}$. Formally, $S = \cG$, the HLS groupoid defined in Example \ref{ex:HLS} but we view $S$ as an inverse semigroup in the following way. The product is given by $q_m(t).q_n(u) = q_{m\wedge n}(tu)$ where $m\wedge n$ is the smallest of the two elements $m,n$. We set $q_m(t)^* = q_m(t^{-1})$. The set $E_S$ of idempotents is $\set{q_m(e) : m \in \widehat{\N}}$ that we identify with $\widehat{\N}$. The product of two idempotents is given by $m.n = m\wedge n$. The spectrum $X$ is the set $\set{\chi_m : m \in \widehat{\N}}$ where $\chi_m(n) = 1$ if and only if $m\leq n$. It is homeomorphic to the compact space $\widehat{\N}$.  The groupoid $\cG_S$ associated with $S$ is the space of equivalence classes of pairs $\big(q_m(t),\chi_k\big)$ with $k\leq m$,  where $\big(q_m(t),\chi_k\big) \sim \big(q_n(u),\chi_k'\big)$ if and only if $k = k'$ and $q_k(t) = q_k(u)$. The map sending the class of $\big(q_m(t),\chi_k\big)$ to $q_k(t)$ is an isomorphism of topological groupoids from $\cG_S$ onto the HLS-groupoid $\cG$. 
 
  The maximal group homomorphism image of $S$ is the finite group $\Gamma/N_0$ and $\sigma : S \to S/\sigma = \Gamma/N_0$  is $q_m(t) \mapsto q_0(t)$. It is not idempotent pure.  Note that $S$ has a zero if and only if $N_0 = \Gamma$, the zero being then $q_0(e)$.

Let us observe that $S = \sqcup_{k\in \widehat{\N}} \Gamma/N_k$ is a Clifford semigroup. 

 Since $S/\sigma$ is amenable, this semigroup $S$ is left amenable by a result of Duncan and Namioka (see \cite[Proposition A.0.5]{Pat}). If $\Gamma = \F_2$ and the sequence $(N_k)_{k}$ is the one defined by Willett in \cite{Wil15}, then $S$ has the weak containment property but $C_{r}^*(S)$ is not nuclear. Moreover, $S$ is a Clifford semigroup for which not every subgroup is amenable, although it has the weak containment property.
 
On the other hand, if we realize $\F_2$ as a finite index subgroup of $SL(2,\Z)$, and choose $N_k$ to be the intersection with $\F_2$ of the kernel of the reduction map $SL(2,\Z) \to SL(2,\Z/2^k\Z)$, then the corresponding HLS-groupoid has not the weak containment property, as observed in \cite[Remarks 2.9]{Wil15}. Hence, the left amenability of $S$ does not imply its weak containment property in general.
\end{ex}

 \begin{rem}\label{thm:strongEunit} Let $S$ be a strongly $E^*$-unitary inverse semigroup, and let $\psi : S^\times \to G$ be an idempotent pure morphism. 
 Assume that $G$ is amenable. Then we have $C^*(S) = C^*_{r}(S)$ and $C^*_{r}(S)$ is nuclear. Indeed,  by Proposition \ref{prop:strongEunit}, the groupoid $\cG_S$ is associated to a partial action of $G$ on the spectrum of $S$.  If $G$ is amenable, then $\cG_S$ is amenable  (see \cite{RW})  and therefore our statement  holds.
 
 Assume  that $G$ is exact. Is is true that the weak containment property of $S$  is equivalent to the nuclearity of $C^*_{r}(S)$?
\end{rem}

\subsection{Weak containment for semigroups embedded in groups}\label{sect:alg_sg}
In this section we consider a discrete group $G$ and a sub-semigroup $P$ which contains the unit $e$. We denote by $\lambda$ the left regular representation of $G$, and for $p\in P$ we denote by $V_p: \ell^2(P) \to \ell^2(P)$ the isometry given by $V_p \delta_q = \delta_{pq}$. The {\it reduced $C^*$-algebra} $C^*_{r}(P)$ of $P$ is the $C^*$-algebra generated by the isometries $V_p$, $p\in P$. 

The right definition of the full $C^*$-algebra of $P$ is more speculative. The universal $C^*$-algebra generated by elements $v_p$, $p\in P$, such that $v^*_{p} v_p = 1$ and $v_p v_q = v_{pq}$ for every $p,q\in P$, is too big.  For instance Murphy proved that, for the commutative semigroup $\N^2$, this universal $C^*$-algebra is not   nuclear \cite{Mur}. A reasonable definition for the {\it full $C^*$-algebra} of $P$ was introduced by Xin Li in \cite[Definition 2.2]{Li12} and a variant  in \cite[Definition 3.2]{Li12}. It is this variant (denoted $C^*_{s}(P)$ in \cite{Li12}) that we adopt as the definition of the full $C^*$-algebra of $P$ in the sequel, and we denote it  $C^*(P)$. By \cite[Proposition 3.3.1]{Nor14}, $C^*(P)$ can be defined as $C^*_{0}(S(P))$. Let us recall (see \cite[Lemma 3.2.2]{Nor14}) that the inverse semigroup $S(P)$ is canonically isomorphic to the inverse semigroup of partial isometries 
$$V(P) = \set{V_{p_1}^*V_{q_1} \cdots V_{p_n}^*V_{q_n} : n\in \N, p_i,q_i\in P}.$$
Let us also recall that there is a surjective homomorphism $h: C^*_{r,0}(S(P)) \to C^*_{r}(P)$ such that
$h(\pi_2(\ell_p)) = V_p$ for $p\in P$ (see \cite[Lemma 3.2.12]{Nor14}). Therefore we have the following situation
$$C^*(P) = C^*_{0}(S(P)) \equiv C^*(\cG_{S(P)}) \stackrel{\pi_{2,0}}{-\!\!\!\twoheadrightarrow} C^*_{r,0}(S(P))\equiv C^*_{r}(\cG_{S(P)}) \stackrel{h}{-\!\!\!\twoheadrightarrow} C^*_{r}(P).$$

\begin{defn} We say that $P$ has the {\it weak containment property} if $C^*(P) = C^*_{r}(P)$.
\end{defn}
Note that the weak containment property of $P$ implies the weak containment property of $S(P)$.

\begin{defn}\label{def:qlog} Let $(P,G)$ as above and assume in addition that $P\cap P^{-1} = \set{e}$. Then we define on $G$ a partial order by setting $x\leq y$ if $x^{-1}y\in P$.
We say that $(P,G)$ is a {\it quasi-lattice ordered group} if for every $g\in G$, we have either $P\cap gP = \emptyset$  or $P\cap gP = rP$  for some $r\in P$ (equivalently, every pair of elements in $G$ having a common upper bound has a least common upper bound (see \cite[Lemma 7]{CL} for more on this)).
\end{defn}

The Toeplitz condition for $(P,G)$ is equivalent to the following property:  for every $g\in G$ such that $E_P \lambda_g E_P \not= 0$, there exist $p_1,\dots, p_n, q_1,\dots, q_n \in P$ such that $E_P \lambda_g E_P = V_{p_1}^*V_{q_1}\dots V_{p_n}^*V_{q_n}$ (see for instance the proof of \cite[Proposition 4.1]{Nor15}).

Quasi-lattice ordered semigroups and semigroups $(P,G)$ such that $G = P^{-1}P$ satisfy the Toeplitz condition (see \cite[\S 8]{Li13}). Xin Li has also introduced an important condition for $P$, he called {\it independence} (\cite[Definition 2.26]{Li12}). We will not describe it here. We only note that when $P$ is contained in a group, this condition is equivalent to the injectivity of $h$  (see \cite[Theorem 3.2.14]{Nor14}) and that it is satisfied for quasi-lattice ordered groups (see \cite[Lemma 28]{Li12}).

\begin{prop}\label{prop:indep} Let $P$ be a sub-semigroup of a group $G$ with $e\in P$. Assume that $P$ satisfies   the independence condition. Then the nuclearity of $C^*_{r}(P)$ implies the weak containment property for $P$, i.e., $C^*(P) =C^*_{r}(P)$.
\end{prop}

\begin{proof} Assume that $C^*_{r}(P)$ is nuclear. Since $C^*_{r}(P) = C^*_{r}(\cG_{S(P)})$, we see that the groupoid $\cG_{S(P)}$ is amenable. It follows that $S(P)$ (and thus $P$) has  the weak containment property.
\end{proof}
 
\begin{prop}\label{prop:left_amen} Let $P$ be a sub-semigroup of a group $G$ with $e\in P$.
\begin{itemize}
\item[(1)] If $G$ is amenable, then  $C^*_{r}(P)$ is nuclear.
\item[(2)] If $P$ is left amenable, then $PP^{-1}$ is an amenable subgroup of $G$ and $C^*_{r}(P)$ is nuclear.
\end{itemize}
\end{prop}
 
 \begin{proof} Assume that $G$ is amenable. By Proposition \ref{prop:Eunit} and Proposition \ref{prop:strongEunit} the groupoid $\cG_{S(P)}$ is amenable since it is isomorphic to the groupoid defined by a partial action of $G$. It follows that $C^*_{r}(\cG_{S(P)})$ is nuclear as well as its quotient $C^*_{r}(P)$.
 
 Suppose now that $P$ is left amenable. Then it is left reversible ({\it i.e.,} $pP \cap qP \not =\emptyset$ for all $p,q\in P$) and therefore $G'=PP^{-1}$ is an amenable subgroup of $G$ (see \cite[Propositions 1.23, 1.27]{Pat88}). To see that $C^*_{r}(\cG_{S(P)})$ is nuclear, we apply the first part of the proof.
 \end{proof}

\begin{rem}\label{thm:semi} Let $P$ be a sub-semigroup of a group $G$, containing the unit $e$. 
Assume that $C^*$-algebra $C^*_{r}(P)$ is nuclear and that $P$ satisfies the independence condition. Then  $C^*(P) =C^*_{r}(P)$ by Proposition \ref{prop:indep}. Is the converse true when $G$ is exact?
\end{rem}

\bibliographystyle{plain}

\end{document}